\documentclass[a4paper,10pt]{amsart}

\usepackage[latin1]{inputenc}
\usepackage[T1]{fontenc}
\usepackage[english]{babel}

\usepackage{mathrsfs}
\usepackage{amscd}
\usepackage{amsfonts}
\usepackage{amsmath}
\usepackage{amssymb}
\usepackage{amstext}
\usepackage{amsthm}
\usepackage{amsbsy}

\usepackage{xspace}
\usepackage[all]{xy}
\usepackage{graphicx}
\usepackage{url}
\usepackage{latexsym}

\usepackage{xspace}
\usepackage[all]{xy}
\usepackage{graphicx}
\usepackage{url}
\usepackage{latexsym}

\newtheorem{theo}{Theorem}[section]

\theoremstyle{definition}
\newtheorem{dfn}[theo]{Definition}
\newtheorem{rem}[theo]{Remark}
\newtheorem{ex}[theo]{Example}

\newcommand{\Z}{\ensuremath{\mathbb{Z}}}

\newcommand{\R}{\ensuremath{\mathbb{R}}}
\newcommand{\C}{\ensuremath{\mathbb{C}}}

\newcommand{\M}{\ensuremath{\mathcal{M}}}
\newcommand{\Rs}{\ensuremath{\mathcal{R}}}

\newcommand{\vs}{\vspace{0.2cm}}

\DeclareMathOperator{\Aut}{Aut}

\DeclareMathOperator{\GL}{GL} 
\DeclareMathOperator{\SL}{SL} 
\DeclareMathOperator{\SO}{SO} 
\DeclareMathOperator{\PSL}{PSL}
\DeclareMathOperator{\UT}{UT}

\begin{document}

\author{Annalisa Conversano  } 
 
\title[Groups definable  in o-minimal structures: a diagram]{Groups definable  in
o-minimal structures: \\ various properties and a diagram}

\address{Massey University Auckland, New Zealand} 

\email{a.conversano@massey.ac.nz}

%\urladdr{http://www.math.uni-konstanz.de/~conversa/homeng.html} 

%\date{Research supported in part by the the Hausdorff Institute in Bonn.}

%

%\date{\today}

\maketitle
\vspace{-1.4cm}

\begin{abstract} 
We present a diagram surveying equivalence or strict implication for properties of different nature (algebraic, model theoretic, topological, etc.) about  groups definable in o-minimal structures. All results are well-known and an extensive bibliography is provided.
\end{abstract}

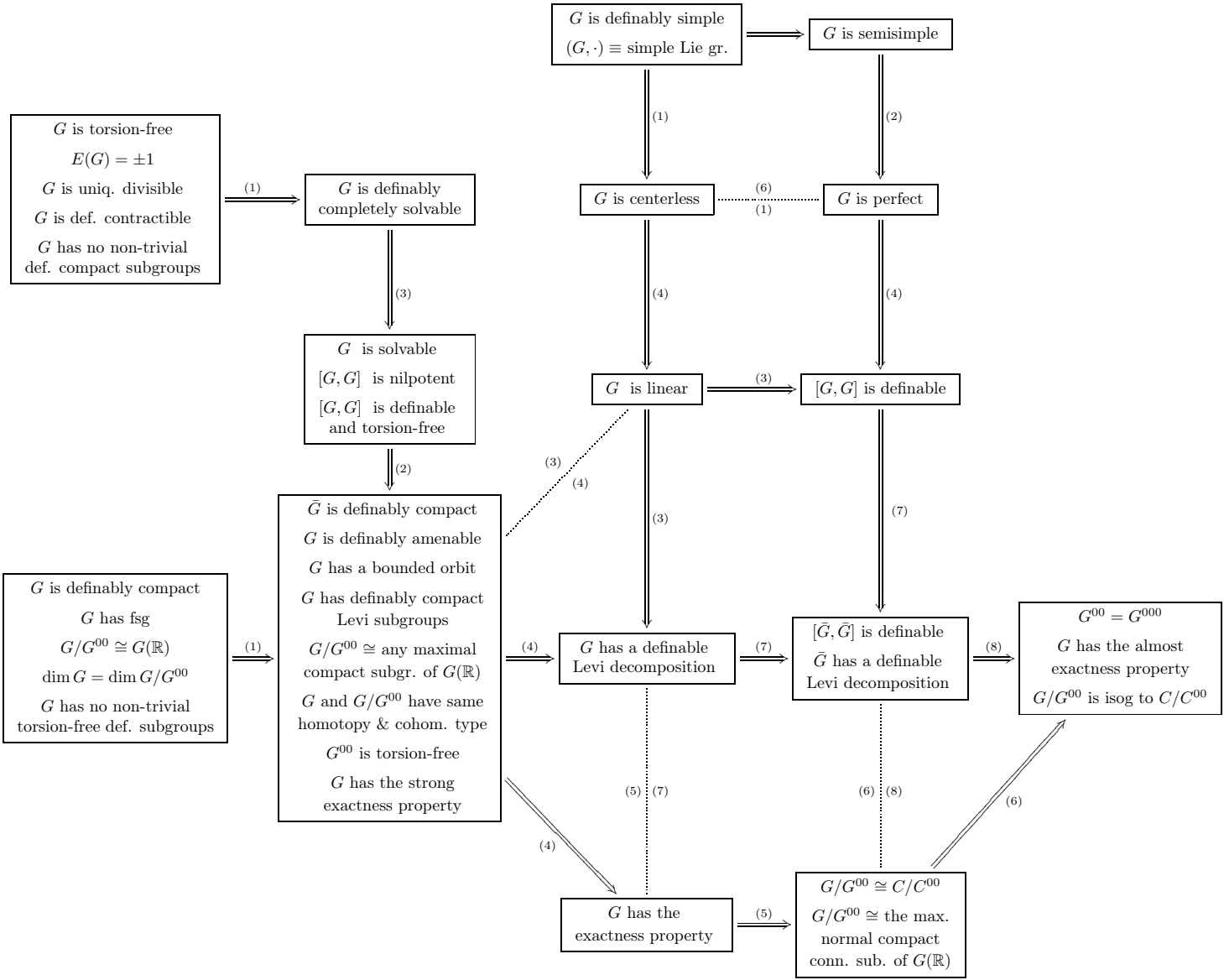
\thispagestyle{empty}

%\newpage

 %\hspace{-.5cm} 
 
%{\Huge \bf Some properties of definably connected groups in o-minimal structures}
% 

%Below is a diagram showing connections between several properties of a definably connected group $G$ definable in a (sufficiently saturated) o-minimal expansion of a real closed field $\mathcal{R}$. See \cite{Otero} for a nice introduction to the topic. 

%No group $G$ can belong to both boxes connected by a slanted double arrow. \\

%The smallest type-definable subgroup of bounded index is $G^{00}$ and the smallest invariant subgroup of bounded index is $G^{000}$. The quotient $G/G^{00}$ is a compact Lie group, when equipped with the logic topology \cite{}.\\
 
% \[
% \hspace{-1.5cm}
%\begin{array}{cccc}
%  & \quad \mathcal{N}(G):=\text{the maximal normal definable}  & \quad \bar G:=G/ \mathcal{N}(G) \quad &  \quad C:=\text{the maximal normal definable} \\
%&  \text{ torsion-free subgroup of }G & \quad &   \text{ compact connected subgroup of }\bar G
%\end{array}
%\]

%\thispagestyle{empty}

%\newsavebox{\boks}
%\savebox{\boks}{G \text{is definably compact}\linebreak G has fsg} 
% \vspace{.5cm} 
\[
\hspace{-2.7cm}
\scalebox{.75}{
\xymatrix{
 &  & \framebox{\begin{tabular}{c}$G$ is definably simple 
 \\[.2cm] 
 $(G, \cdot) \equiv$ simple Lie gr.  \end{tabular}}  \ar@{=>}[r]  \ar@{=>}[d]^{(1)}    & 
\framebox{\begin{tabular}{c}  $G$ is semisimple 
\end{tabular}}\ar@{=>}[d]^{(2)} \\
 \framebox{\begin{tabular}{c}
$G$ is  torsion-free\\[.2cm]
%$N$ \& $G/N$ are torsion-free\\
%for a (any) def $N \lhd G$\\[.2cm]
$E(G) = \pm 1$\\[.2cm]
$G$ is uniq. divisible \\[.2cm]
$G$ \text{is def. contractible}\\[.2cm]
$G$ has no non-trivial \\
def. compact subgroups\\
\end{tabular}
} \ar@{=>}[r]^{(1)}  
 & 
\framebox{\begin{tabular}{c} 
$G$ is definably \\ 
completely solvable \end{tabular}} \ar@{=>}[d]^{(3)} &
\framebox{\begin{tabular}{c} $G$ is centerless  \end{tabular}} \ar@{=>}[d]^{(4)}  \ar@{.}[r]^{(6)}_{(1)}
& \framebox{\begin{tabular}{c} $G$ is perfect  \end{tabular}} \ar@{=>}[d]^{(4)}   
 \\
%\framebox{$G$ \text{ is torsion-free}} \ar[d]  & 
%  & \framebox{$G$ \text{ is torsion-free}} \ar[d] \\
&   
\framebox{\begin{tabular}{c}
$G$ \text{ is solvable}\\[.2cm]
$[G ,G]$ \text{ is nilpotent}\\[.2cm]
$[G ,G]$ \text{ is definable}\\
\text{and torsion-free}  
\end{tabular}
} \ar@{=>}[d] ^-{(2)}   & 
 \framebox{\begin{tabular}{c} $G$ \text{ is linear}  \end{tabular}} \ar@{=>}[d]^--{(3)} \ar@{=>}[r]^{(3)} 
&
\framebox{\begin{tabular}{c} $[G, G]$ is definable   \end{tabular}}
\ar@{=>}[d]^--{\ (7)} 
\\ 
\framebox{\begin{tabular}{c}
$G$ is definably compact\\[.2cm] 
%$N$ \& $G/N$ are def. comp.\\
%for a (any) def $N \lhd G$\\[.2cm]
$G$ has fsg \\[.2cm]
%$G$ is comp.dom. by $G/G^{00}$ \\[.2cm]
$G/G^{00}\cong G(\mathbb R)$\\ [.2cm]
$\dim G = \dim G/G^{00} $\\[.2cm] 
$G$ has no non-trivial\\
torsion-free def. subgroups\\
\end{tabular}}\ar@{=>}[r]^-{(1)}  
& \framebox{
\begin{tabular}{c}
$\bar G$ is definably compact\\[.2cm]
$G$ is definably amenable\\[.2cm]
$G$ has a bounded orbit\\[.2cm]
$G$ has definably compact \\[.05cm]
Levi subgroups \\[.2cm]
$G/G^{00}\cong$ any maximal \\[.05cm]
compact subgr. of $G(\mathbb R)$\\[.2cm]
$G$ and $G/G^{00}$ have same \\[.05cm]
homotopy \& cohom. type\\[.2cm]
$G^{00}$ is torsion-free\\[.2cm]
$G$ has the strong\\[.05cm]
exactness property\\
\end{tabular}}  \ar@{=>}[r]^-{(4)}\ar@{=>}[dr]_-{(4)} \ar@{.}[ur]^-{(3)}_-{(4)} &    
\framebox{\begin{tabular}{c}
$G$ has a definable\\
Levi decomposition 
\end{tabular}
}\ar@{=>}[r]^{(7)}\ar@{.}[d]^{(7)} &
\framebox{\begin{tabular}{c}
$[\bar G,\bar G]$ is definable\\[.2cm]
$\bar G$ has a definable\\[.05cm]
Levi decomposition\\
\end{tabular}
}\ar@{=>}[r]^-{(8)}\ar@{.}[d]^{(8)} &
\framebox{\begin{tabular}{c}
$G^{00}=G^{000}$\\[.2cm]
$G$ has the almost \\[.05cm]
exactness property\\[.2cm]
$G/G^{00}$ is isog to $C/C^{00}$
\end{tabular}} \\
&  
%\framebox{\begin{tabular}{c}
%$(G, \cdot) \equiv (G/G^{00}, \cdot)$\\[.2cm]
%$G/Z(G)$ is def. comp.\\[.05cm]
%\& $Z(G)^0$ has torsion.
%\end{tabular}} 
&
\framebox{\begin{tabular}{c}
$G$ has the\\[.05cm]
exactness property\\
\end{tabular}
}\ar@{=>}[r]^-{(5)}
\ar@{}[u]^{(5)} &
\framebox{\begin{tabular}{c}
$G/G^{00}\cong C/C^{00}$\\[.2cm]
$G/G^{00}\cong$ the max.\\[.05cm]
normal compact \\[.05cm]
conn. sub. of $G(\mathbb R)$
\end{tabular}}\ar@{=>}[ur]_{(6)}\ar@{}[u]^{(6)}}} 
\]

\def\R{\mathbb R}
\def\Z{\mathbb Z}

%\vspace{0.5cm}
%
%{\bf Some counterexamples:}
%\newpage
%\noindent 
\section{Introduction}

Groups definable in o-minimal structures have been studied by many authors in the last 30 years and include algebraic groups over algebraically closed fields of characteristic 0, semialgebraic groups over real closed fields, important classes of real Lie groups such as abelian groups, compact groups and linear semisimple groups. See \cite{Otero} for a nice introduction to the topic.  

The previous diagram shows connections between several properties of a definably connected group $G$ definable in a (sufficiently saturated) o-minimal expansion \M\ of a real closed field $\mathcal{R}$ (although most results are known in arbitrary o-minimal structures -- or o-minimal expansions of groups).   \\

Notation in the diagram is as follows:
we denote by $\mathcal{N}(G)$ the maximal normal definable torsion-free subgroup of $G$, by $\bar{G}$ the quotient $G/ \mathcal{N}(G)$, and by $C$ the maximal normal definably connected definably compact subgroup of $\bar G$. Subgroups $\mathcal{N}(G)$ and $C$ are proved to exist in \cite[Proposition 2.6]{us1}. 

$E(G)$ denotes the \emph{o-minimal Euler characteristic} of $G$ (see Section 2).  

Assuming $G$ is definable with parameters over the reals, $G(\R)$
denotes the corresponding real Lie group (see \cite[Prop 2.5 \& Rem 2.6]{Pillay - groups}).

By a beautiful conjecture of Pillay from \cite{PC} (see \cite{PPC} for a survey about its solution), the quotient $G/G^{00}$ (of a definably connected group $G$ by the smallest type-definable subgroup of bounded index $G^{00}$) is a compact Lie group, when equipped with the logic topology \cite{BOPP}. Finally, we denote by $G^{000}$ the smallest subgroup of bounded index in $G$ which is $\Aut(\M)$-invariant.\\

Properties in each box are equivalent. An arrow indicates that a property in a box implies the properties in the other one.
The dotted lines show that there is no implication in either direction.  
All implications are strict. \\

The diagram is explained as follows:
Section 2 introduces the left-most column of torsion-free and definably compact groups. In Section 3 solvable groups $G$ and their quotient $\bar{G} = G/\mathcal{N}(G)$ are discussed. Section 4 covers the right part of the diagram. Finally, Section 5 provides counterexamples for the implications that do not hold.  \\

Further work on definable groups outside the scope of this diagram can be found in \cite{BBO19, Cartan, BO10, BPP, me3, El08, EHP18, EPR, Edmundo-Terzo, GPP, Nash, NIPII, J1, J2, MMT00, NPR, Alf, PY, T00}. In recent years, the investigation has been extended by several authors to the wider class  of locally definable groups. See, for instance, \cite{BEP19, BEM13, E06, EP12, EP13}. \\

\textbf{Aknowledgments.} This survey has been mostly prepared during the 
``Logic and Algorithms in Group Theory'' Program (September 3 - December 20, 2018) at the Hausdorff Institute in Bonn. Many thanks to the organizers for the kind invitation and the great scientific program.

%=====================
%\vs
\section{Torsion-free and definably compact groups}

%\noindent
The left side of the diagram shows classes of torsion-free and definably compact groups. Roughly speaking, torsion-free groups resemble triangular groups of upper triangular matrices, and definably compact groups are very closely related to closed subgroups of orthogonal groups. Definable torsion-free groups have been studied, for instance, in \cite{BO16, COS, PeSta} and definably compact groups, much more extensively, in \cite{BB12, B17, EMPRT, NIPI, Ma, OP09, Peterzil-Pillay}. \\

If $\mathcal{P}$ is a cell decomposition of a definable set $X$, \emph{the o-minimal Euler characteristic} $E(X)$ is defined as the number of even-dimensional cells in 
$\mathcal{P}$ minus the number of odd-dimensional cells in $\mathcal{P}$, and it does not depend on $\mathcal{P}$ (see \cite{Lou}, Chapter 4). 
Strzebonski proved in \cite[Prop 2.5]{Strzebonski} that a definable group $G$ is {\bf torsion-free} if and only if $E(G) = \pm 1$, and in  \cite[Prop 4.1]{Strzebonski} he deduced that such groups are uniquely divisible. By \cite[2.4 and 2.11]
{PeSta} torsion-free groups $G$ are also definably connected and solvable. In fact $G$ is definably completely solvable (or triangular) \cite[Theo 4.4]{COS}. Namely, $G$ contains a chain of normal (in $G$) definably connected subgroups 
\[
\{e\} = G_0 < G_1 < \dots < G_{n-1} < G_n = G
\]

\noindent
where $\dim G_i = i$. Note that since 1-dimensional definably connected groups are abelian \cite{Razenj}, definable groups with such a chain are solvable. Abelian definably compact groups are not, in general, definably completely solvable. See, for instance, Examples 5.1 and 5.2 in \cite{PS}. By \cite[Prop 2.5]{me1} and  \cite[Cor 5.7]{PeSta}, a definable group is definably contractible if and only if it is torsion-free. \\

In $\aleph_0$-saturated ordered structures, only sequences that are eventually constant converge, so the standard notion of compactness by sequences is not very useful in a model-theoretic context. Definable compactness has been introduced by Peterzil and Steinhorn in \cite[Def 1.1]{PS} as a better analogue to compactness.  A definable set $X$ is {\bf definably compact} if 
for every definable continuous function $f \colon (a, b) \to X$, the limits of $f(x)$, as $x$ tends to $a$ or to $b$, exist in $X$. 
%every definable curve in it is completable. 
Over any o-minimal structure, if $X$ has the topology induced by the order of the ambient structure, this is equivalent to say that $X$ is closed and bounded \cite[Theo 2.1]{PS}. Thus over the reals, this coincides with the usual notion of compactness. The equivalent condition using definable open coverings is given in \cite[Cor 2.3]{Peterzil-Pillay}.

%========check solvable groups!================
By \cite[Theo 8.1 \& Rem3 pg.588]{NIPI}, $G$ is definably compact if and only if  $G$ has \emph{fsg} (finitely satisfiable generics). That is, there is a global type $p(x)$ and a small model $\Rs_0$ such that for every $g \in G$ the left translate $gp = \{\varphi(x): \varphi(g^{-1}) \in p\}$ is finitely satisfiable in $\Rs_0$.

If $G$ is definably compact, then the o-minimal dimension of $G$ as a definable set coincide with the dimension of $G/G^{00}$ as a Lie group \cite[Theo 8.1]{NIPI} (see \cite{AA17} for another proof, and \cite{EMPRT} for the case of an arbitrary o-minimal structure). 
Moreover, when $G$ is defined over the reals, $G(\R)$  is Lie isomorphic to $G/G^{00}$ via the standard part map \cite{NIPI}. 

Conversely, if $G$ is not definably compact, then by \cite[Theo 1.2]{PS} $G$ contains infinite definable torsion-free subgroups $H$, for which $H = H^{00} \subset G^{00}$ \cite[2.4]{us2}, and so $\dim G/G^{00} < \dim G$. Moreover, when $G$ is defined over the reals, the Lie group $G(\R)$ contains the torsion-free closed subgroup $H(\R)$ (where $H$ is the torsion-free definable subgroup mentioned above). Since closed subgroups of compact groups are compact too, and compact Lie groups have torsion, it follows that in this case $G(\R)$ cannot be Lie isomorphic to the compact $G/G^{00}$.\\

Both classes of definably compact and torsion-free definable groups are closed by definable subgroups and definable quotients \cite[2.3]{PeSta}.  

\noindent
Since definably compact groups have torsion \cite{Edmundo-Otero, Pe02}, it follows that definably compact and torsion-free definable subgroups of a definable group always have trivial intersection.  If a definable group is not definably compact, then it contains a 1-dimensional torsion-free definable subgroup \cite[Theo 1.2]{PS}. Therefore, for both classes, the condition of being torsion-free or definably compact is equivalent to not having any non-trivial definable subgroup of the other class.

 Every definable linear group $G$ can be decomposed into a product $G = KH$ of a (maximal) definable torsion-free subgroup $H$ and a (maximal) definably compact subgroup $K$ \cite[Theo 4.1]{me2}. If $G$ is not linear, $G$ may not have maximal definably compact subgroups, but a similar decomposition holds where $K$ is abstractly compact. That is, it is isomorphic to a definably compact subgroup of $G/\mathcal{N}(G)$ \cite[Theo 1.5]{me2}. On the other hand, maximal definable torsion-free subgroups always exist \cite[Cor 2.4]{me1} and they are conjugate to each other \cite[Theo 3.26]{PeSta-mu}.  
%======================
%\vs
\section{Solvable groups}

%\noindent
Solvable definable groups have been first studied by Edmundo in \cite{Edmundo}. As observed in \cite[Prop 2.2]{us1}, the quotient $\bar{G}$ of a definable solvable group $G$ by its maximal normal definable torsion-free subgroup $\mathcal{N}(G)$ is definably compact. Moreover, if $G$ is not definably compact then $\mathcal{N}(G)$ is infinite (unlike semisimple groups that are not definably compact). We give below a direct proof of both facts (see \cite[Theo 2.5.1]{Conversano-thesis}):

%=============
\begin{theo}\label{solvquozcomp}
Let $G$ be a solvable definable group and let $N = \mathcal{N}(G)$ be its maximal normal definable torsion-free subgroup. If $G$ is not definably compact, then $N$ is infinite and $G/N$ is a definably compact group.  
\end{theo}

\begin{proof}
Because $N \subseteq G^0$  and $G/G^0$ is finite (so definably compact), we can suppose $G$ is definably connected.
We proceed by induction on $n = \dim G$. The case $n = 1$ is obvious, so let $\dim G = n > 1$.

If $G$ is abelian, the theorem can be extracted from \cite[2.6]{PeSta}. The argument is that if $G/N$ is not definably compact, then by \cite[Theo 1.2]{PS} there is a definable 1-dimensional torsion-free subgroup $H$ in $G/N$,
and the pull-back of $H$ in $G$ is a definable torsion-free subgroup of $G$, in contradiction with the maximality of $N$.  

Let $G$ be now non-abelian.
Since $G$ is solvable and definably connected, there is a normal solvable definable subgroup $S < G$
such that $G/S$ is abelian and infinite (definability of $S$ follows from \cite[1.17]{PPSI}). We distinguish the cases where
$S$ is definably compact and where $S$ is not definably compact.

\begin{itemize}

\item  If $S$ is definably compact then $G/S$ is not. By the abelian case,
the maximal normal definable torsion-free subgroup $N_1$ of $G/S$ is infinite
and $(G/S)/N_1$ is definably compact. If $\pi \colon G \to G/S$ is the canonical projection, 
let $N^{\prime} = \pi^{-1}(N_1)$. By \cite[Lemma 2.3]{me1} the definable exact sequence
\[
1\ \longrightarrow\ S \ \stackrel{i}{\longrightarrow}\ N^{\prime} \stackrel{\pi}{\longrightarrow}\ N_1\ \longrightarrow\ 1  
\]

\vs \noindent
splits definably in a direct product, thus $G$ contains a definable subgroup $N$ definably isomorphic to $N_1$ such that $N^{\prime} = S \times N$. Since $S \cong N^{\prime}/N$ which is definably compact, it follows that $N$ is the maximal normal definable torsion-free subgroup of $N^{\prime}$, and  it is normal in $G$ as well.
% (Lemma \ref{tormaxnor}). 

To show that $G/N$ is definably compact, it is enough to provide a normal
definable subgroup of $G/N$ which is definably compact, such that quotienting by it
we obtain a definably compact group. One such subgroup is $N^{\prime}/N$ which is definably isomorphic
to $S$, and the quotient  
%This is because $N^{\prime}/N \cong S$ and so it is a definably compact subgroup of $G/N$, and the quotient 
%$(G/N)/(N^{\prime}/N) \cong G/N^{\prime}$ which is definably compact as well: it 
$(G/N)/(N^{\prime}/N)$ is definably isomorphic to $(G/S)/N_1$, as the following diagram shows by ``the $3 \times 3$ lemma''.%: see Mac Lane homology.

\[
\xymatrix{
& & 1 \ar[d] & 1 \ar[d] & \\
%1 \ar[r] & S \ar[r] & \mathop{N^{\prime}}\limits^{pp} \ar[r] \ar[d]& N_1 \ar[r] \ar[d] & 1\\
1 \ar[r] & S \ar[r] & N^{\prime} \ar[r] \ar[d]& N_1 \ar[r] \ar[d] & 1\\
1 \ar[r] & S \ar[r] \ar@{<->}[u] & G \ar[r] \ar[d] & G/S \ar[r] \ar[d] & 1\\
& & G/N^{\prime} \ar@{<->}[r] \ar[d] & (G/S)/N_1 \ar[d] & \\
& & 1 & 1 & 
 }
\]

\vs
Thus $N$ is the maximal normal definable torsion-free subgroup of $G$, it is infinite and the theorem is proved for the case where $S$ is definably compact.  \\

%If $S^{\prime}$ is the image of $S$ in $G/N$ by the canonical projection, then $(G/N)/S^{\prime}$ is definably isomorphic to $(G/S)/N_1$ definably compact and therefore, again by \ref{quocomp}, $N$ is the maximal normal definable torsion-free subgroup of $G$ and $G/N$ is definably compact because $S^{\prime}$ is definably isomorphic to $S$ which is definably compact and $(G/N)/S^{\prime}$ is definably compact.\\

\item If $S$ is not definably compact, then by induction the maximal
normal definable torsion-free subgroup $N_1$ of $S$ is infinite
(possibly $N_1 = S$) and $S/N_1$ is definably compact. 
Note that $N_1$ is normal in $G$ as well.
% showed in \ref{tormaxnor}. 
If $G/N_1$ is definably compact then $N_1$ is the maximal normal definable torsion-free subgroup of $G$
and we are done. Otherwise, again by induction, its infinite maximal normal definable torsion-free subgroup
$N_2$ is such that $(G/N_1)/N_2$ is
definably compact. 

Let $N$ be the pull-back in $G$ of $N_2$. Note that $N$
%If $\pi \colon G \to G/N_1$ is the canonical
%projection, then $N = \pi^{-1}(N_2) < G$ 
is torsion-free and $G/N$ is definably isomorphic to $(G/N_1)/N_2$ which is definably compact.
Hence $N$ is the maximal normal definable torsion-free subgroup
of $G$ and it is infinite, since it contains $N_1$.
\end{itemize} \vspace{-.4cm}
\end{proof}

\noindent 
While any connected solvable real Lie group splits into a product of 1-dimensional connected subgroups \cite[Lem 3.6]{Iwasawa}, definable solvable groups with torsion  are not, in general, definably completely solvable. See Example 5.3.  \\

\noindent
If $G$ is solvable, then $[G,G]$ is nilpotent by \cite[Theo 6.9]{Edmundo}. Since $G/[G,G]$ is abelian for any group $G$, the converse is obvious.\\

\noindent
If $G$ is solvable, then $[G, G]$ is definable by \cite[Theo 1.3]{BJO}. Moreover, definably connected definably compact solvable groups are abelian by \cite[Cor 5.4]{Peterzil-Starchenko1}.  Therefore, if $G$ is solvable, then $\bar{G}$ is abelian and $[G, G] \subseteq \mathcal{N}(G)$ is torsion-free. Conversely, if $[G, G]$ is definable and torsion-free, then it is solvable by \cite[2.11]{PeSta}, and $G$ is solvable as well. The condition of $[G, G]$ being definable is necessary to conclude that $G$ is solvable, since $\widetilde{\SL}_2(\R)$ is torsion-free (see Example 5.7 where $[G, G] = \widetilde{\SL}_2(\R)$). \\

 %================================
%\vs
\section{Definable amenability, exactness, and connected components}

Recall that $G$ is said to be \emph{definably amenable} if it has a left invariant Keisler measure and $G$ \emph{has bounded orbit} if there is some complete type $p \in S_G(\mathcal{M}) $ whose stabilizer $Stab(p) = \{g \in G : gp = p\}$ has bounded index in $G$. In \cite{NIPI} groups with $fsg$ in complete NIP theories are shown to be both definably amenable and with bounded orbit, by lifting the Haar measure of the compact Lie group $G/G^{00}$ to a left invariant Keisler measure on $G$, making use of a global generic type $p$, whose stabilizer is $G^{00}$. The two classes of groups are indeed  shown to be the same in the o-minimal context \cite[Cor 4.12]{us1}, and to coincide with the class of groups $G$ such that $\bar{G} = G/\mathcal{N}(G)$ is definably compact \cite[Prop 4.6]{us1}. In short, if $\bar{G}$ is definably compact, then $G$ is definably amenable because    torsion-free and definably compact groups are definably amenable. Conversely, if  $\bar{G}$ is not definably compact, then there is a definable quotient of $G$ that is a definably simple not definably compact group, and such groups are not definably amenable (such as $\PSL_2(\R)$ \cite[Rem 5.2]{NIPI}).

Before discussing the remaining equivalent conditions for $\bar{G}$ to be definably compact, let us consider the upper right part of the diagram. \\

A non-abelian definable group is said to be {\bf definably simple} if it does not have any non-trivial normal definable subgroup. By \cite[Theo 5.1]{PPSIII}, definably simple groups are exactly the definable groups that are elementarily equivalent to a (non-abelian) simple real Lie group.

As the center $Z(G)$ is a normal definable subgroup of any definable group $G$, it follows that definably simple groups are centerless. \\

Given any definable group $G$, the quotient  $G/Z(G)$ can be definably embedded  in some $\GL_n(\Rs)$ through the adjoint representation \cite[Cor 3.3]{OPP96}. Therefore {\bf centerless groups} are linearizable.\\

%On the opposite side of solvable groups are semisimple groups. 
An infinite definable group is said to be {\bf semisimple} if it does not have any infinite abelian (or, equivalently, solvable) -- definable or not -- normal subgroup.  Semisimple definable groups have been studied by Peterzil, Pillay and Starchenko in \cite{PPSI, PPSII}. They prove that given a semisimple group $G$ definable in an arbitrary o-minimal structure \M, the quotient by the center $G/Z(G)$ is a direct product of definably simple groups $H_i$, and each $H_i$ is definably isomorphic to a definable subgroup of $\GL_n(\Rs_i)$, for some real closed field $\Rs_i$ definable in \M. \\

It is well-known that every definable group $G$ has a maximal normal definably connected  solvable subgroup $R$ (called \emph{the solvable radical}) and the quotient $G/R$ is a semisimple definable group. Since a proof does not appear in the literature, as far as we know, we add it below together with the proof of the existence of a maximal solvable subgroup:

\begin{rem}
 Let $G$ be a definably connected definable group. Then $G$ has a unique normal solvable definably connected subgroup $R$ such that $G/R$ is trivial or semisimple. The subgroup $R$ is the maximal normal solvable definably connected subgroup of $G$. If $G/R$ is semisimple and $\pi \colon G \to G/R$ is the canonical projection, then $\pi^{-1}(Z(G/R))$ is the maximal normal solvable subgroup of $G$.
 \end{rem}
 
 \begin{proof}
 By induction on $n = \dim G$. If $G$ is not semisimple, let $A < G$ be an infinite normal definable subgroup. If $G/A$ is semisimple, take $R = A^0$. Otherwise, by induction there is a normal solvable definably connected subgroup $S < G/A$ such that $(G/A)/S$ is semisimple. Then take $R$ to be the definably connected component of the identity of the pre-image of $S$ in $G$.
 
 If $S$ is another normal solvable definably connected subgroup of $G$, then $RS/R$ is a normal solvable definably connected subgroup of the semisimple group $G/R$. Therefore $RS = R$ and $S \subseteq R$.
 
If $S \subsetneq R$, then $\dim S < \dim R$  and $R/S$ is an infinite solvable definable subgroup of $G/S$. Therefore $G/S$ cannot be semisimple and $R$ is unique.
 
 Let $H = \pi^{-1}(Z(G/R))$. As mentioned above, $G/H$ is a direct product of definably simple groups and does not contain any solvable (definable or not) normal subgroup. Therefore $H$ is the maximal normal solvable subgroup of $G$.
 \end{proof}

Semisimple groups $G$ are {\bf perfect} by \cite[3.1]{HPP}. That is, $G$ is equal to its commutator subgroup $[G, G]$ or, equivalently, $G$ does not have any proper abelian quotient.

 By \cite[Theo 4.5]{PPSIII}, {\bf linear} groups $G$ have a definable Levi decomposition. That is, $G$ contains (maximal) semisimple definable subgroups $S$ (all conjugate) such that $G = RS$, where $R$ is the solvable radical. See \cite[Theo 1.1]{us3} for a Levi decomposition of an arbitrary definably connected group $G$, where $S$ is in general a countable union of definable sets.  
 
Very recently Baro proved in \cite[Theo 3.1]{Baro19} that the commutator subgroup of a linear group is definable. In general, this is not the case. See Examples 5.7 and 5.8, where the commutator subgroup of $G$ is isomorphic to the universal cover of $\rm SL_2(\R)$. These are also examples of groups without a definable Levi decomposition. Another definable group with no definable Levi subgroups is Example 5.9, whereas its commutator subgroup is definable. On the other hand, we expect that definability of Levi subgroups implies definability of the commutator subgroup.

\begin{rem} \label{remark}
If $\bar{G}$ is definably compact, then $G$ has a definable Levi decomposition.  
\end{rem}

\begin{proof}
If $\bar{G}$ is definably compact, then by \cite[Cor 6.4]{HPP}, $\bar{G} = Z(\bar{G}) \cdot [\bar{G}, \bar{G}]$ is a definable Levi decomposition of $\bar{G}$. Let $H$ be the pull-back of $[\bar{G}, \bar{G}]$ in $G$. So $H$ is a definable extension of a definably compact semisimple group by a definable torsion group $\mathcal{N}(G)$. The extension splits definably by \cite[Prop 5.1]{us3} and $H = \mathcal{N}(G) \rtimes S$, for some definably compact semisimple definable subgroup $S$. As $G/H \subset Z(\bar{G})$ is abelian, then clearly $S$ is a definable Levi subgroup of $G$.
 \end{proof}

If $[G, G]$ {\bf is definable}, then clearly $[\bar{G}, \bar{G}]$ {\bf is definable} as well. 

\begin{rem}
$[\bar{G}, \bar{G}]$ is definable  if and only if $\bar{G}$ has a definable Levi decomposition. 
\end{rem}

\begin{proof}
By Theorem \ref{solvquozcomp}, $\mathcal{N}(\bar{G}) = \{e\}$ and the solvable radical $\bar{R}$ of $\bar{G}$ is definably compact. By  \cite[Theo 4.4]{Peterzil-Starchenko1} $\bar{R}$ is abelian. To  see that 
$\bar{R}$ is central in $\bar{G}$, let $s \colon \bar{G}/\bar{R} \to \bar{G}$ be a definable section of the canonical projection $\pi \colon \bar{G} \to \bar{G}/\bar{R}$.

For every $g \in \bar{G}$, the conjugation map $f_g \colon \bar{R} \to \bar{R}$ mapping $a \mapsto gag^{-1}$ is a definable automorphisms of $\bar{R}$, so there is a homomorphism $\Phi \colon \bar{G} \to \Aut(\bar{R})$, given by $g \mapsto f_g$ such that $\bar{R} \subseteq \ker \Phi$, being $\bar{R}$ abelian. Thus $\Phi $ induces the definable homomorphism
\begin{align*}
\varphi \colon & \bar{G}/\bar{R} \longrightarrow \Aut (\bar{R}) \\
& \ \  x \mapsto (a \mapsto s(x)a s(x)^{-1})  
\end{align*}
which does not depend on the choice of the section $s$. Since $\bar{R}$ has no definable families of definable automorphisms by  \cite[Cor 5.3]{Peterzil-Starchenko1}, it follows that 
 %Let $\varphi \colon G/N \to \Aut (N)$ be a definable homomorphism as in the proof of Proposition \ref{compact-torus-central}. 
  $\varphi(\bar{G}/\bar{R}) = \{e\}$, and so $\bar{R} \subseteq Z(\bar{G})$.   
%  (A different proof of the more general fact that any normal abstract torus is central in a definably connected group is given in \cite[?]{me-JC}.)
%  
By \cite[Lem 3.2]{us3} $\bar{G} = Z(\bar{G}) [\bar{G}, \bar{G}]$  and $[\bar{G}, \bar{G}]$ coincides with the (unique) Levi subgroup of $\bar{G}$. Hence $[\bar{G}, \bar{G}]$ is definable if and only if $\bar{G}$ has a definable Levi decomposition. 
\end{proof}

Finally, by \cite[Prop 2.6]{us2}, whenever $\bar{G}$ has a definable Levi decomposition, then $G^{00} = G^{000}$.\\

We now go back to the class of groups $G$ such that $\bar{G}$ is definably compact.
 In Remark \ref{remark} we observed that $G$ has a definable Levi decomposition $G = RS$, and clearly \emph{Levi subgroups $S$ are definably compact}.
 
 Conversely, if $G = RS$ and $S$ is definably compact, then $\bar{G}$ must be definably compact, as $R/\mathcal{N}(G)$ is definably compact by Theorem \ref{solvquozcomp}.\\

When $G$ is defined over the reals, then $\bar{G}$ is definably compact if and only if $G/G^{00}$ is isomorphic to a maximal compact subgroup of $G(\R)$ \cite[Prop 2.10]{us2}.  Moreover, $\bar{G}$ definably compact is also equivalent for $G^{00}$ to be torsion-free and
for $G$ and $G/G^{00}$ to have same homotopy and cohomology types \cite{Conversano-thesis}.
 A last equivalent condition regards exactness of sequences. Given an exact sequence of definably connected groups
\[
1 \to H \to G \to Q \to 1
\] 
the induced sequences $1 \to H^{00} \to G^{00} \to Q^{00} \to 1$ and $1 \to H^{000} \to G^{000} \to Q^{000} \to 1$ turn out to be exact if and only if $H \cap G^{00} = H^{00}$ and $H \cap G^{000} = H^{000}$ (see \cite[Lemma 2.1]{us2}). So we have the following definition from \cite{us2}:

\begin{dfn}
Let $G$ be a definable group. We say that

\begin{enumerate}
\item $G$ has \emph{the almost exactness property} if for every normal definable subgroup $H$ of $G$, $H^{00}$ has finite index in $G^{00} \cap H$.

\item $G$ has \emph{the exactness property} if for every normal definable subgroup $H$ of $G$, $H^{00} = G^{00} \cap H$.

\item $G$ has \emph{the strong exactness property} if for every definable subgroup $H$ of $G$, $H^{00} = G^{00} \cap H$.
\end{enumerate}
\end{dfn}

When $G$ is definably compact, then $G$ {\bf has the strong exactness property} by \cite{B07}. By \cite[Theo 4.11]{us2} $G$ has the strong exactness property if and only if $\bar{G}$ is definably compact. \\

If $G$ {\bf has the exactness property}, then $G/G^{00}$ is isomorphic to $C/C^{00}$ by \cite[Prop 4.10]{us2}. If $G$ is defined over the reals, this is equivalent for $G/G^{00}$ to be isomorphic to the maximal normal compact connected subgroup of $G(\R)$ \cite{us2}. \\

Finally, $G$ {\bf has the almost exactness property} if and only if $G^{00} = G^{000}$ 
(\cite[Theo 4.4]{us2}) if and only if $G/G^{00}$ is isogenous to $C/C^{00}$ (\cite[Rem 4.7]{us2}).

%===========================
%\vs
\section{Counterexamples} 

In this last section we describe 9 semialgebraic groups proving that all implications in the diagram are strict.  The enumeration coincides with the number on the corresponding arrows -- or dotted lines -- in the diagram.
%\vs
%\begin{enumerate}
%\item $ \mathcal{R}^2 \rtimes {\rm SO}_2$ \vs
%\item $ \mathcal{R}^3 \rtimes {\rm SO}_3$ \vs
%\item $ {\rm SO}_2 \times {\rm SL}_2$ \vs
%\item $ \mathcal{R} \times_{\Z} \widetilde {\rm SL}_2$  \vs
%\item $ {\rm SO}_2 \times_{\Z} \widetilde {\rm SL}_2,\; \Z \cap {\rm SO}_2^{00} = \{e\}$   \vs
%\item $ {\rm SO}_2 \times_{\Z} \widetilde {\rm SL}_2,\; \Z\subseteq {\rm SO}_2^{00}$   \vs
%\item $ {\rm SU}_2 \times_{\{\pm I\}}{\rm SL}_2 $ \vs
%\item $ {\rm SO}_2 \times_{\{\pm I\}} {\rm SL}_2$ \vs
%
%\end{enumerate}
%--------1----------------
\begin{ex}{\bf :} \fbox{$\mathcal{R}^2 \rtimes {\rm SO}_2(\Rs)$} 

\vspace{.1cm}
${\rm SO}_2(\mathcal{R})$ denotes the special group of othogonal matrices $2 \times 2$ with coefficients in $\mathcal{R}$. That is,
\[
{\rm SO}_2(\mathcal{R}) = \left \{ \begin{pmatrix}
a & -b \\
b & a
\end{pmatrix} : a, b \in \Rs ,\ a^2 + b^2 = 1
\right\}
\]

${\rm SO}_2(\mathcal{R})$ is an abelian 1-dimensional definably compact group (abstractly) isomorphic to the unit circle.

Let now $G$ be the following group of matrices $3 \times 3$:

\[
G = \left \{ 
\begin{pmatrix}
a & -b & x \\
b & a & y \\
0 & 0 & 1
\end{pmatrix} : a, b, x, y \in \Rs,  a^2 + b^2 = 1
\right \}
\]

The group $G$ is definably isomorphic to the semidirect product $\mathcal{R}^2 \rtimes {\rm SO}_2(\Rs)$ where the action is the matrix multiplication of ${\rm SO}_2(\Rs)$ on ($\Rs^2, +)$.

Note that $\mathcal{N}(G) \cong \Rs^2$ and $G/\mathcal{N}(G) \cong {\rm SO}_2(\Rs)$ is definably compact, but $G$ is not definably compact nor torsion-free.

Moreover, $G$ is centerless but it is not perfect nor definably simple.
\end{ex}

%-----------------------2------------------
\vspace{.2cm}
\begin{ex}{\bf :} \fbox{$\mathcal{R}^3 \rtimes {\rm SO}_3(\Rs)$} 

\vspace{.1cm}
A similar group $G$ of matrices $4 \times 4$ can be obtained considering the action given by the matrix multiplication of ${\rm SO}_3(\Rs)$, the special orthogonal group of matrices $3 \times 3$ with coefficients in \Rs, on $(\Rs^3, +)$. 

Note that $\mathcal{N}(G) \cong \Rs^3$ and $G/\mathcal{N}(G) \cong {\rm SO}_3(\Rs)$ is a definably compact definably simple group. Thus $G$ is not solvable.

Moreover, $G$ is a perfect group that is not semisimple.
\end{ex}

%-----------------old 9 -----new 3--------------------
\vspace{.2cm}
\begin{ex} \label{ps99}{\bf :} \fbox{$n$-dimensional torus}

\vspace{.1cm}
This is Example 5.2 in \cite{PS}. Given $\mathcal{B} = \{v_1, \dots, v_n\}$ a set of linearly independent vectors in $\R^n$, let $L = v_1\Z+ \dots + v_n\Z$ be the integral lattice generated by $\mathcal{B}$, and $E$ the equivalence on $\R^n$ induced by $L$. That is, for each $a, b \in \R^n$
\[
a\, E\, b \quad \Longleftrightarrow \quad a - b \in L
\]

\noindent
For every bounded box $D$ containing the fundamental parallelogram of $L$, there is a finite sublattice $L' \subset L$ such that for all $a, b \in D$ we have
\[
a\, E\, b \quad \Longleftrightarrow \quad a - b \in L'
\]
Thus, even though the quotient group $(\R^n, +)/L$ is not definable, it is abstractly isomorphic to the definable group $G = (S, \oplus)$, where $S \subset D$ is a definable $n$-dimensional set of representatives of $(\R^n, +)/L$ and the group operation $\oplus$ on $S$ is defined as
\[
x \oplus y = z \quad \Longleftrightarrow \quad (x + y) \ E\ z 
\]

\vspace{.1cm}
Peterzil and Steinhorn prove in \cite{PS} that $\mathcal{B}$ can be chosen so that the proper definable subgroups of $G$ are all finite. When this is the case and $n > 1$, $G$ is an example of an abelian group that is not definably completely solvable.

Moreover $G$ is definably compact, but it is not definably isomorphic to a definable subgroup of the general linear group, as abelian definably compact linear groups split in a direct product of 1-dimensional definable subgroups by \cite[Lem 3.9]{PPSIII}. 
\end{ex}

%---------------------old 3---new-4-----------------
\vspace{.2cm}
\begin{ex}{\bf :} \fbox{$ {\rm SO}_2(\Rs) \times {\rm SL}_2(\Rs)$} 

\vspace{.1cm}
${\rm SL}_2(\Rs)$ denotes the special linear group $2 \times 2$, that is the group of matrices with coefficients in \Rs\ with determinant equal to 1. It is a semisimple definably connected group that is not definably compact.

Let $G$ be the direct product of ${\rm SO}_2 (\Rs)$ by ${\rm SL}_2(\Rs)$. The group $G$ is linear and is not centerless. Moreover $G$ has a definable Levi decomposition (the unique Levi subgroup is ${\rm SL}_2$) and $\bar{G} = G $ is not definably compact. Its commutator subgroup $[G, G]$ is again ${\rm SL}_2(\Rs)$, so it is definable, but $G$ is not perfect. Finally, note that $G^{00} = {\rm SO}_2(\Rs)^{00} \times {\rm SL}_2(\Rs)$, so $G$ has the exactness property by \cite[4.3]{us2}.
\end{ex}

%--------------------------5----------------
\vspace{.2cm}
\begin{ex}{\bf :} \fbox{$ {\rm SO}_2(\Rs) \times_{\{\pm I\}}{\rm SL}_2(\Rs)$}

\vspace{.2cm}
Let now consider the group $G$ obtained from ${\rm SO}_2 (\Rs) \times {\rm SL}_2(\Rs)$ above by identifying the common central subgroup $H = \{I, -I\}$. A way to define $G$ is the following:

Fixed a definable choice of representatives $s \colon \PSL_2(\Rs) \to \SL_2(\Rs)$ (for instance, $s(A)$ the matrix $[a_{ij}]$ such that $a_{11} >~0$ or $a_{11} = 0$ and $a_{12} > 0$), let $G = \SO_2(\Rs) \times \PSL_2(\Rs)$ be with the group operation given by

\[
(X_1, A_1) \ast (X_2, A_2) = 
\begin{cases}
\: \: (X_1X_2,\ A_1A_2) & \text{if $s(A_1A_2) = s(A_1)s(A_2)$,}\\
(-X_1X_2,\ A_1A_2) & \text{otherwise}.
\end{cases}
\]

\vs%\noindent 
$G$ contains normal definable subgroups $R = \SO_2(\Rs) \times \{I\}$  and $S = \{\pm I\} \times \PSL_2(\Rs)$ ($S$ is definably isomorphic to $\SL_2(\Rs)$), such that $RS = G$ and $R \cap S = \{(I, I), (-I, I)\}$. Note that 
$R^{00} = \SO_2(\Rs)^{00} \times \{I\}, 
G^{00}  = \pm \SO_2(\Rs)^{00} \times \PSL_2(\Rs)$, and
$G^{00} \cap R = \pm \SO_2(\Rs)^{00} \times \{I\}$. Therefore $R^{00}$ is properly contained in $G^{00} \cap R$ and $G$ does not have the exactness property. Moreover $G/G^{00} \cong \SO_2(\Rs)/\SO_2(\Rs)^{00} \cong C/C^{00}$.
\end{ex}

%-------------------------6-----------------da scambiare con 7--------------------
\vspace{.2cm}
\begin{ex}{\bf :} \fbox{$ {\rm SU}_2(\mathcal{K})  \times_{\{\pm I\}}{\rm SL}_2(\Rs) $}

\vspace{.1cm}
Denoted by $\mathcal{K} = \Rs(\sqrt{-1})$ the algebraic closure of the real closed field \Rs, every semialgebraic subgroup of the general linear group ${\rm GL}_n(\mathcal{K})$ can be viewed as a definable subgroup of ${\rm GL}_{2n}(\Rs)$. One such group is the special unitary group:

\[
{\rm SU}_n(\mathcal{K}) = \{X \in {\rm GL}_n(\mathcal{K}): X \bar{X}^T = I, \  \det X = 1 \}
\]

\vspace{0.2cm}
When $n = 2$, ${\rm SU}_2(\mathcal{K})$ is the universal cover of ${\rm SO}_3(\Rs)$ wih kernel $\{\pm I\}$, therefore it is a semisimple definably compact group. 

Let now $G$ be the amalgamated direct product of ${\rm SU}_2(\mathcal{K})$ and ${\rm SL}_2(\Rs)$ over the common central subgroup $\{\pm I\}$, obtained definably as the previous example. Note that $\mathcal{N}(G) = \{e\}$ and the maximal normal definably compact definably connected subgroup of $\bar{G} = G$ is $C = {\rm SU}_2(\mathcal{K}) $. Therefore $G/G^{00} \cong {\rm SO}_3(\R)$,  while $C/C^{00} \cong {\rm SU}_2(\C)$.
Moreover, $G$ is a perfect group that is not centerless.
\end{ex}
 
%---------------------------7---------------
\vspace{.2cm} \label{buffo} 
\begin{ex}{\bf :} \fbox{$ \R \times_{\Z} \widetilde{\SL}_2(\R) $}

\vspace{.2cm}
\noindent 
We now describe a small modification of \cite[Example 2.10]{us1} and \cite[Example 3.1.7]{Conversano-thesis}. The idea is to define a semialgebraic group that is isomorphic to the amalgamated direct product of $(\R, +)$ and the universal cover of $\SL_2(\R)$ over the common subgroup $(\Z, +) \cong Z(\widetilde{\SL}_2(\R))$.

Let $\pi \colon \widetilde{\SL}_2(\R) \to \SL_2(\R)$ be the universal cover of $\SL_2(\R)$ 
%(which is also the o-minimal universal cover of $PSL(2, \R)$) 
and fix $s \colon \SL_2(\R) \to \widetilde{\SL}_2(\R)$ a definable section. Recall that $\widetilde{\SL}_2(\R)$ is a semisimple Lie group with infinite center $Z(\widetilde{\SL}_2(\R)) = \ker \pi$ and $\pi$ is a homomorphism. Thus the image of the 2-cocycle $
h_s \colon \SL_2(\R) \times \SL_2(\R)\to \widetilde{\SL}_2(\R)$ given by $h_s(x_1, x_2) = s(x_1)s(x_2)s(x_1x_2)^{-1}$ is contained in $\ker \pi$. Hrushovski, Peterzil and Pillay prove in  \cite[8.5]{HPP} that it is actually a definable map and takes only finitely many values. 
%The function $h_s(x_1, x_2) = s(x_1)s(x_2)s(x_1x_2)^{-1}$ is a definable map $PSL(2, \R) \times PSL(2, \R) \to \ker \pi$. 
Note that $\ker \pi$ is isomorphic (as an abstract group) to $(\Z, +)$. Fixed a generator $v$ of $\ker \pi$, if $h_s(x_1, x_2) = kv,\ k \in \Z$, with abuse of notation we write $h_s(x_1, x_2)$ meaning the corresponding $k \in \Z$.\\
%with abuse of notation we will write $h_s(x_1, x_2) = k \in \Z$ meaning $h_s(x_1, x_2) = kv$.  
Now let $G = \R\ \times\ \SL_2(\R)$. Fixed $a \neq 0$, consider the definable group operation on $G$ given by

$$
(t_1, x_1) \ast (t_2, x_2) = (t_1 + t_2 + h_s(x_1, x_2)a,\ x_1x_2).
$$

\vs  
Then $(G, \ast)$ is a semialgebraic group such that $[G, G]$ is not definable and with no definable Levi subgroups. 
 %($G$ is the free product of $([0, 1[, \oplus)$ and $(\widetilde{SL}_2, \cdot)$ with amalgameted subgroup $(\Z, +)$).\\
%$G = [0, 1[ \cdot \widetilde{SL}_2$ with amalgamation $[0, 1[\ \cap\ \widetilde{SL}_2 = \langle (a, I) \rangle = \langle (0, v) \rangle \cong (\Z, +)$.
To see this, note that the center of $G$ (whose connected component coincides with the solvable radical) is the subgroup $Z = Z(G) = \R\ \times\ \{\pm I\}$.
%$Z^{00} = \bigcap_{n \in \N}([0, 1/n[ \cup [1-1/n, 1[)$. We want to show that \\
%if $a \in Z^{00}$ then $G^{00} = Z^{00} \cap PSL(2, R)$, if $a \not \in Z^{00}$ then $G = G^{00}$.\\
%\begin{claim} 
%If $G$ is a definable group such that $G = G^{00}$, then the o-minimal universal cover $\widetilde{G}$ of $G$ has no type-definable subgroups of bounded index.
%\end{claim}
%
%\begin{proof}
%Let $\pi \colon \widetilde{G} \to G$ the universal cover map. Suppose $\widetilde{X} = \bigcap_{i \in I}X_i$ a type definable subgroup of bounded index in $\widetilde{X}$. $\pi(\widetilde{X}) = \bigcap_{i \in I} \pi(X_i)$ is a type-definable (the universal cover map is definable restricted to every definable set) subgroup ($\pi$ is an homomorphism) of bounded index ($\pi$ is surjective and $\widetilde{X}$ is of bounded index in $\widetilde{G}$). Then $\pi(\widetilde{X}) = G$.  
%\end{proof}
We can identify the subgroup $\langle a \rangle \times\ \SL_2(\R)$ with $\widetilde{\SL}_2(\R)$, which is isomorphic to it by construction. 
Every element $g$ of $G$ is a product $g = zx$, for some $z \in Z$ and $x \in \widetilde{\SL}_2(\R)$.
Therefore every commutator of $G$ is a commutator of $\widetilde{\SL}_2(\R)$, and $[G, G] = \widetilde{\SL}_2(\R)$ is the unique Levi subgroup of $G$.
%: $[g_1, g_2] = [z_1x_1, z_2x_2] = z_1x_1z_2x_2x_1^{-1}z_1^{-1}x_2^{-1}z_2^{-1} = x_1x_2x_1^{-1}x_2^{-1} = [x_1, x_2]$. Thus $[G, G] = [\widetilde{SL}_2, \widetilde{SL}_2] = \widetilde{SL}_2$, since every semisimple Lie group is equal to its derived subgroup. But $\widetilde{SL}_2$ is not definable in $\bar{\R}$ (it is a semisimple group with infinite center), then $[G, G]$ is not definable. Observe that for every semisimple group $S$ such that $G = ZS$ then $[G, G] = [S, S] = S$. Therefore such an $S$ can not be definable in $\bar{\R}$.  
However, $G/\mathcal{N}(G) = \SL_2(\R)$, so $\bar{G} = [\bar{G}, \bar{G}]$.
\end{ex}

%-----------------------------8-------------
\vspace{.2cm}
\begin{ex}{\bf :}  \fbox{$ \SO_2(\Rs) \times_{\Z} \widetilde{\SL}_2(\Rs) $, $\Z \subset \SO_2^{00}(\Rs)$}

\vspace{.2cm}
Let us now consider the semialgebraic group $G$ in a suffiently saturated elementary extension $\Rs$ of the reals, obtained as the previous example by replacing $\R$ with $\SO_2(\Rs)$ and by taking $a \in \SO_2(\Rs)^{00}$, so that $\mathcal{N}(G) = \{e\}$ and $G = \bar{G}$. 

As before, $[G, G] = \langle a \rangle \times\ \SL_2(\Rs)$ is not definable, and $G^{00} = G^{000} = \SO_2(\Rs)^{00} \times \SL_2(\Rs)$. Therefore $G/G^{00} \cong \SO_2(\Rs)/\SO_2^{00}(\Rs) \cong C/C^{00}$.
\end{ex}

%========================
 \vspace{.2cm}
\begin{ex}{\bf :} \fbox{$ \UT_3(\R) \times_{\Z} \widetilde{\SL}_2(\R) $}

\vspace{.2cm}
Let $\UT_3(\R)$ be the group of real unipotent matrices (that is, upper triangular matrices with 1's on the diagonal) $3 \times 3$. Given $h_s \colon \SL_2(\R) \times \SL_2(\R)\to \widetilde{\SL}_2(\R)$ as in Example 5.7, one can define on $G = \UT_3(\R) \times \SL_2(\R)$  the following   group operation:

\[
(A, X) \ast (B, Y) = \left (AB + \begin{pmatrix}
0 & 0 & h_s(X, Y) \\
0 & 0 & 0 \\
0 & 0 & 0
\end{pmatrix}, XY
\right )
\]

\vs \noindent
so that $Z(G)^0 = Z(\UT_3(\R)) \cong (\R, +)$ and $H = \Z \times \SL_2(\R) \cong  \widetilde{\SL}_2(\R)\, \lhd G$ is the unique Levi subgroup of $G$. 
Note that $[G, G] = Z(G)^0 \times \SL_2(\R)$ is isomorphic to Example 5.7.
Therefore the commutator subgroup of $G$ is definable, while Levi subgroups are not. 

 In \cite{CM} it is proved that the nilpotent Lie group $G/H$ interprets the real field expanded with a predicate for the integers, and therefore it interprets every real Lie group.
 \end{ex}

%------------------------------------------
 
\vs

\end{document}